\documentclass[11pt,twoside]{amsart}
\usepackage{amsmath, amsthm, amscd, amsfonts, amssymb, graphicx, color}
\usepackage[bookmarksnumbered, colorlinks, plainpages]{hyperref}
\usepackage{pgf,tikz}
\usetikzlibrary{arrows}
\usepackage[utf8]{inputenc}
\usepackage{pstricks-add}
\input{mathrsfs.sty}
\newtheorem{thm}{Theorem}[section]
\newtheorem{cor}[thm]{Corollary}
\newtheorem{lem}[thm]{Lemma}

\newtheorem{defn}[thm]{Definition}
\newtheorem{rem}[thm]{\bf{Remark}}

\numberwithin{equation}{section}

\begin{document}


\title{Graphs whose all maximal induced forests are of the same order}
\author{Reza Jafarpour-Golzari}
\address{Department of Mathematics, Payame Noor University, P.O.BOX 19395-3697 Tehran, Iran; Department of Mathematics, Institute for Advanced Studies
in Basic Science (IASBS), P.O.Box 45195-1159, Zanjan, Iran}

\email{r.golzary@iasbs.ac.ir}

\thanks{{\scriptsize
\hskip -0.4 true cm MSC(2010): Primary: 05C70; Secondary: 05C38, 05C75.
\\
\newline\indent{\scriptsize}}}

\maketitle


\begin{abstract}
In this paper, a new concept in graphs namely well-f-coveredness is introduced. We characterize all graphs with such property, whose maximum induced forests are of boundary order. Also we prove several propositions concerning with obtaining large well-f-covered graphs from smaller ones. By the way, some interesting classes of wel
l-f-covered graphs are characterized.
\end{abstract}

\vskip 0.8 true cm


\section{\bf Introduction}
\vskip 0.4 true cm
In the sequel, all graphs are finite and simple. Also we refer to \cite{Wes} for any backgrounds in graph theory.

A graph $G$ is said to be well-covered, whenever all maximal independent vertex sets in $G$ are of the same size. In other words, $G$ is well-covered, whenever all its maximal induced edgeless subgraphs are of the same size. The concept of well-coveredness was introduced by Plummer (See \cite{Plu1}) and many mathematicians have investigated about it and its applications in recent decades (for instance, see \cite{Fin2}, \cite{Plu2}, and \cite{Ran}). Some various classes of well-covered graphs have been characterized and behavior of well-coveredness under some graph operations has been studied (see \cite{Jaf}, \cite{Fin1}, \cite{Top}, and \cite{Vil}, for instance). Motivated by concept of well-covered graphs, we introduce a new concept in graphs, named well-f-covered graph.

After offering some elementary examples of graphs satisfying well-f-coveredness property, we study the relationship between this property and well-coveredness. Then we characterize all well-f-covered graphs whose maximum induced forests are of boundary order (Theorems 5.1 and 5.2).
Also we prove several propositions which offer some approaches to obtain larger well-f-covered graphs from smaller ones.

Since the concepts of well-coveredness and forest both have a wide range of applications in graph theory and many other branches of mathematics and some other sciences, it seems that the concept of well-f-coveredness can have a good prospect of applications in the future.

\vskip 0.8 true cm

\section{\bf Preliminaries}
\vskip 0.4 true cm
Let $G$ be a graph and $X$ be a subset of $V(G)$. The subgraph of $G$ induced by $X$, is denoted by $G[X]$. For convenience, in the continuation of this article, we will assume that subgraph means induced subgraph. Also for a vertex $v$ in $G$, the set of all neighbors of $v$ is denoted by $N(v)$.

In the graph $G$, an edge $e$ is said to be bridge, if its removal increases the number of connected components. An edge $e=\{x,y\}$ is called a pendant edge if at least one of $x$ and $y$ be of degree 1. Adding a pendant edge to $G$, means adding a vertex as $x$ and an edge $e=\{x,y\}$ to $G$, where $y$ is a vertex of $G$, and removing a pendant edge, means removing a vertex $x$ of degree 1, and the edge incident with $x$. A cord for a cycle $C$, is an edge $e=\{x,y\}$ where $x$ and $y$ are two non-adjacent distinct vertices of $C$.

A bipartite graph $G$ with parts $X$ and $Y$, is said to be complete bipartite, if each vertex of $X$ be adjacent with every vertex of $Y$. Such a graph, if $X$ and $Y$ be of sizes $r$ and $s$, respectively, is denoted by $K_{r,s}$. A wheel of order $n$, denoted by $W_{n}$, is a graph obtained by taking a cycle of size $n-1$, and adding a vertex which is adjacent with all others. In a connected graph, the length of a shortest path between two vertices $x$ and $y$, is called distance of $x$ and $y$, and is denoted by $d(x,y)$.

A maximal forest in the graph $G$, is a subgraph $F$ which is a forest and there is not any vertex $x$ such that $G[V(F)\cup \{x\}]$ is a forest, as well. Such a subgraph is a maximum forest in $G$, whenever its order be greater than or equal to the order of any other forest in $G$.

A clique in the graph $G$, is a set $Q\subseteq V(G)$ in which every two distinct vertices are adjacent. A matching in $G$,
 is a set $M\subseteq E(G)$ in which every two distinct edges are non-adjacent.

 In a graph $G$, a subset $X$ of $V(G)$ is said to be an independent vertex set, whenever no two vertices in $X$ be adjacent. The size of a maximum independent set in $G$ is called the independence number of $G$ and is denoted by $\alpha (G)$. The graph $G$ is said to be well-covered whenever all its maximal independent sets be of the same size \cite{Plu1}.

\vskip 0.8 true cm

\section{\bf Well-f-covered graphs}
\vskip 0.4 true cm
For a graph $G$, we name the order of any maximum forest of $G$, the forest number of $G$ and we denote it by $f(G)$.

Now we introduce our new concept.

\begin{defn}
A graph $G$ is said to be well-f-covered, if all its maximal forests be of the same order $f(G)$.
\end{defn}

Clearly each forests itself is well-f-covered. Also every cycle of size $n$, $C_{n}$, is well-f-covered and $f(C_{n})=n-1$.

Each $K_{n}$ is well-f-covered and $f(K_{n})=2$, for $n\geq 2$, meanwhile $f(K_{1})=1$.

If a graph $G$ be well-f-covered, adding or removing a bridge dose not invalidate well-f-coveredness and also the forest number (note that a bridge is not contained in any cycle). Therefore for checking that a graph is well-f-covered or not, one can remove all its bridges, first. 

\vskip 0.8 true cm

\section{\bf The relationship to well-coveredness}
\vskip 0.4 true cm
Well-f-coveredness and well-coveredness are two independent properties in graph. Consider the following four graphs:
\begin{center}
\definecolor{ududff}{rgb}{0.30196078431372547,0.30196078431372547,1.}
\begin{tikzpicture}[line cap=round,line join=round,>=triangle 45,x=1.0cm,y=1.0cm]
\clip(-1.,1.3) rectangle (10.5,6);
\draw [line width=1.2pt] (0.9,4.44)-- (-0.6,4.46);
\draw [line width=1.2pt] (-0.6,4.46)-- (0.1,3.3);
\draw [line width=1.2pt] (0.9,4.44)-- (0.1,3.3);
\draw [line width=1.2pt] (3.26,3.84)-- (2.42,4.46);
\draw [line width=1.2pt] (2.42,4.46)-- (2.4,3.22);
\draw [line width=1.2pt] (3.26,3.84)-- (2.4,3.22);
\draw [line width=1.2pt] (3.26,3.84)-- (4.08,4.44);
\draw [line width=1.2pt] (4.08,4.44)-- (4.04,3.22);
\draw [line width=1.2pt] (3.26,3.84)-- (4.04,3.22);
\draw [line width=1.2pt] (6.62,4.86)-- (5.56,4.86);
\draw [line width=1.2pt] (5.56,4.86)-- (6.6,3.88);
\draw [line width=1.2pt] (6.6,3.88)-- (5.58,3.88);
\draw [line width=1.2pt] (5.58,3.88)-- (6.58,2.88);
\draw [line width=1.2pt] (6.58,2.88)-- (5.58,2.88);
\draw [line width=1.2pt] (8.98,4.9)-- (8.2,4.14);
\draw [line width=1.2pt] (8.2,4.14)-- (8.18,2.8);
\draw [line width=1.2pt] (8.18,2.8)-- (9.68,2.78);
\draw [line width=1.2pt] (9.68,2.78)-- (9.7,4.12);
\draw [line width=1.2pt] (8.98,4.9)-- (9.7,4.12);
\draw [line width=1.2pt] (9.7,4.12)-- (8.2,4.14);
\draw (-0.27,2.6) node[anchor=north west] {$\mathit{G_{1}}$};
\draw (2.92,2.6) node[anchor=north west] {$\mathit{G_{2}}$};
\draw (5.76,2.54) node[anchor=north west] {$\mathit{G_{3}}$};
\draw (8.6,2.52) node[anchor=north west] {$\mathit{G_{4}}$};
\draw (3.98,1.79) node[anchor=north west] {\small{Figure 1}};
\draw (2,4.84) node[anchor=north west] {$\mathit{x}$};
\draw (1.99,3.23) node[anchor=north west] {$\mathit{y}$};
\draw (3.03,4.34) node[anchor=north west] {\textit{z}};
\draw (4.04,4.84) node[anchor=north west] {$\mathit{t}$};
\draw (4.0,3.3) node[anchor=north west] {$\mathit{p}$};
\draw (8.73,5.33) node[anchor=north west] {$\mathit{e}$};
\draw (7.75,4.5) node[anchor=north west] {$\mathit{a}$};
\draw (9.65,4.54) node[anchor=north west] {$\mathit{b}$};
\draw (9.65,2.91) node[anchor=north west] {$\mathit{c}$};
\draw (7.7,3.0) node[anchor=north west] {$\mathit{d}$};
\begin{scriptsize}
\draw [fill=ududff] (0.9,4.44) circle (1.5pt);
\draw [fill=ududff] (-0.6,4.46) circle (1.5pt);
\draw [fill=ududff] (0.1,3.3) circle (1.5pt);
\draw [fill=ududff] (3.26,3.84) circle (1.5pt);
\draw [fill=ududff] (2.42,4.46) circle (1.5pt);
\draw [fill=ududff] (2.4,3.22) circle (1.5pt);
\draw [fill=ududff] (4.08,4.44) circle (1.5pt);
\draw [fill=ududff] (4.04,3.22) circle (1.5pt);
\draw [fill=ududff] (6.62,4.86) circle (1.5pt);
\draw [fill=ududff] (5.56,4.86) circle (1.5pt);
\draw [fill=ududff] (6.6,3.88) circle (1.5pt);
\draw [fill=ududff] (5.58,3.88) circle (1.5pt);
\draw [fill=ududff] (6.58,2.88) circle (1.5pt);
\draw [fill=ududff] (5.58,2.88) circle (1.5pt);
\draw [fill=ududff] (8.98,4.9) circle (1.5pt);
\draw [fill=ududff] (8.2,4.14) circle (1.5pt);
\draw [fill=ududff] (8.18,2.8) circle (1.5pt);
\draw [fill=ududff] (9.68,2.78) circle (1.5pt);
\draw [fill=ududff] (9.7,4.12) circle (1.5pt);
\end{scriptsize}
\end{tikzpicture}
  \end{center}
  \label{Figure 1}
The graph $G_{1}$ is both well-f-covered and well-covered. $G_{2}$ is not well-f-covered (the maximal forests $G_{2}[\{x,y,t,p\}]$ and $G_{2}[\{x,z,p\}]$ are of different orders) nor well-covered. $G_{3}$ is well-f-covered but not well- covered. Finally, $G_{4}$ is not well-f-covered (the maximal forests $G_{4}[\{e,a,d,c\}]$ and $G_{4}[\{a,b,c\}]$ are of different orders) but it is well-covered.

As we saw above, adding or removing a bridge in a well-f-covered graph, does not invalidate well-f-coveredness. But for well-coveredness, the issue is different. The graph:
\begin{center}
\definecolor{ududff}{rgb}{0.30196078431372547,0.30196078431372547,1.}
\begin{tikzpicture}[line cap=round,line join=round,>=triangle 45,x=1.0cm,y=1.0cm]
\clip(1.,2.3) rectangle (4.9,5.5);
\draw [line width=1.2pt] (2.5,4.96)-- (1.72,3.92);
\draw [line width=1.2pt] (1.72,3.92)-- (3.2,3.9);
\draw [line width=1.2pt] (2.5,4.96)-- (3.2,3.9);
\draw (1.24,4.14) node[anchor=north west] {$\mathit{x}$};
\draw (3.16,4.15) node[anchor=north west] {$\mathit{y}$};
\draw (2.27,5.37) node[anchor=north west] {\textit{z}};
\draw (4.22,5.44) node[anchor=north west] {$\mathit{t}$};
\draw (1.9,2.96) node[anchor=north west] {\small{Figure 2}};
\begin{scriptsize}
\draw [fill=ududff] (2.5,4.96) circle (1.5pt);
\draw [fill=ududff] (1.72,3.92) circle (1.5pt);
\draw [fill=ududff] (3.2,3.9) circle (1.5pt);
\draw [fill=ududff] (4.24,5.16) circle (1.5pt);
\end{scriptsize}
\end{tikzpicture}
\end{center}
\label{Figure 2}
is well-covered, but adding a bridge between $z$ and $t$, changes it into a graph which is not well-covered. The graph:
\begin{center}
\definecolor{ududff}{rgb}{0.30196078431372547,0.30196078431372547,1.}
\begin{tikzpicture}[line cap=round,line join=round,>=triangle 45,x=1.0cm,y=1.0cm]
\clip(0.,2.) rectangle (4.5,5.);
\draw [line width=1.2pt] (0.52,3.92)-- (1.7,3.92);
\draw [line width=1.2pt] (1.7,3.92)-- (2.8,3.9);
\draw (1.3,2.96) node[anchor=north west] {\small{Figure 3}};
\draw [line width=1.2pt] (2.8,3.9)-- (3.88,3.88);
\draw (3.09,4.4) node[anchor=north west] {$\mathit{e}$};
\begin{scriptsize}
\draw [fill=ududff] (0.52,3.92) circle (1.5pt);
\draw [fill=ududff] (1.7,3.92) circle (1.5pt);
\draw [fill=ududff] (2.8,3.9) circle (1.5pt);
\draw [fill=ududff] (3.88,3.88) circle (1.5pt);
\end{scriptsize}
\end{tikzpicture}
\end{center}
\label{Figure 3}
is well-covered, and deleting the bridge $e$ invalidates well-coveredness.

\vskip 0.8 true cm

\section{\bf Well-f-covered graphs with boundary forest number}
\vskip 0.4 true cm
The forest number of a graph $G$ is at least 1 and at most the order of $G$. In this section, we characterize all well-f-covered graphs $G$ which their forest number is 1, 2, $|V(G)|$, and $|V(G)|-1$.

It is clear that for a graph $G$, $f(G)=1$, if and only if $G$ be a trivial graph (note that if $f(G)=1$, then $G$ can not contain any edge), and anyway $G$ is well-f-covered.

\begin{thm}
The forest number of a connected graph $G$ is 2, if and only if $G$ be a complete graph, $K_{n}$, with $n\geq 2$, and anyway $G$ is well-f-covered.
\end{thm}

\begin{proof}
If $G$ be $K_{n}$ for $n\geq 2$, it is clear that $f(G)=2$.

Conversely, let $f(G)=2$. We prove that $G$ is $K_{n}$ with $n\geq 2$. The graph $G$ has at least two distinct vertex (because $f(K_{1})=1$). Let $G$ has two non-adjacent distinct vertices $u$ and $v$ (on the contrary). We can suppose that $d(u,v)$ is minimum. Let:
\[(u=)u_{1} u_{2} \ldots u_{n}(=v)\]
be a path of length $d(u,v)$ between $u$ and $v$. We have $n\geq 3$ because $u$ and $v$ is not adjacent. Also by minimality, $u_{1}$ and $u_{3}$ are not adjacent. Therefore $G[\{u_{1}, u_{2}, u_{3}\}]$ is a forest of order 2, and this is a contradiction.
\end{proof}

Finally, one can see that a non-connected graph $G$ is of forest number 2, if and only if it is an empty graph of order 2 (note that if $f(G)=2$, $G$ can not contain any edge because if two vertices $x$ and $y$ of a connected component be adjacent by an edge, for a vertex $w$ of an other component of $G$, the graph $G[\{x,y,w\}]$ is a forest of order 3), and anyway $G$ is well-f-covered. \\

Now we consider the other hand of the spectrum.

Clearly, the forest number of a graph $G$ is $|V(G)|$, if and only if $G$ be a forest, and anyway $G$ is well-f-covered.

\begin{thm}
A graph $G$ is well-f-covered with forest number $|V(G)|-1$, if and only if $G$ has only one cycle.
\end{thm} 

\begin{proof}
Let $G$ has only one cycle. Consider a maximal forest $F$ in $G$. At least one vertex of the unique cycle of $G$ is not contained in $F$. Now $F$ contains all other vertices of $G$, by maximality.

Conversely, let $G$ be well-f-covered with forest number $|V(G)|-1$. Since $G$ is not a forest, $G$ contains at least one cycle. Take a cycle $C$ with minimum length in $G$. We prove that $G$ has not any other cycle. Let $G$ contains another cycle (on the contrary). Take a cycle $C^{\prime}(\neq C)$ which is of minimum length. Since $C$ is of minimum length, $C$ is cordless. Also by minimality, $C^{\prime}$ is a cordless cycle. Suppose that $C$ and $C^{\prime}$ be:
\[x_{1} x_{2} x_{3} \ldots x_{n} x_{n+1} (=x_{1}),\]
\[y_{1} y_{2} y_{3} \ldots y_{m} y_{m+1} (=y_{1}),\]
respectively. There are two cases.

$\mathbf{Case 1.}$ $C$ and $C^{\prime}$ have at most one common vertex. Consider a vertex $x$ in $V(C)- V(C^{\prime})$, and a vertex $y$ in $V(C^{\prime})- V(C)$. Now the subgraph induced by the set $(V(C)\cup V(C^{\prime}))- \{x,y\}$ is a forest of size at most $|V(G)|- 2$, a contradiction.

$\mathbf{Case 2.}$ $C$ and $C^{\prime}$ have at least two common vertices. In this case, since $C$ and $C^{\prime}$ are cordless, it can be seen that there is two distinct cycles $C_{1}$ and $C_{2}$ in $G$ such that $C_{1}$ and $C_{2}$ are common in one edge or only in a few consecutive edges and are not common in any another vertices, and each one of them contains a vertex which is not contained in the other one. Let $x$ be a vertex in $V(C)- V(C^{\prime})$ and $y$ be a vertex in $V(C^{\prime})- V(C)$. Again, $G[(V(C)\cup V(C^{\prime}))- \{x,y\}]$ is a forest of size at most $|V(G)|- 2$, a contradiction.
\end{proof}

\vskip 0.8 true cm

\section{\bf Obtaining larger well-f-covered graphs}
\vskip 0.4 true cm

In this section we prove some lemmas and theorems by which one can achieve larger well-f-covered graphs from smaller ones.

First, we consider the union of two disjoint graphs (two graphs whose vertex sets are disjoint).

\begin{thm}
The union of two disjoint graphs $G$ and $H$ is well-f-covered, if and only if $G$ and $H$ both are well-f-covered, and anyway $f(G\cup H)= f(G)+ f(H)$.
\end{thm} 

\begin{proof}
It can be seen that a subgraph $F$ of $H\cup G$, is a maximal forest of it, if and only if $F$ be the union of a maximal forest of $G$ with a maximal forest of $H$.

Suppose that $G$ and $H$ be well-f-covered. Let $F_{1}$ and $F_{2}$ be two maximal forests of $G\cup H$. The subgraphs $G[V(F_{1})\cap V(G)]$ and $G[V(F_{2})\cap V(G)]$ are two maximal forests of $G$. Since $G$ is well-f-covered, $|V(F_{1})\cap V(G)|= |V(F_{2})\cap V(G)|= f(G)$. A similar argument shows that  $|V(F_{1})\cap V(H)|= |V(F_{2})\cap V(H)|= f(H)$. On the other hand, $V(F_{1})$ is the disjoint union of two sets $V(F_{1})\cap V(G)$ and $V(F_{1})\cap V(H)$, and similarly $V(F_{2})$ is the disjoint union of two sets $V(F_{2})\cap V(G)$ and $V(F_{2})\cap V(H)$. It follows that $|V(F_{1})|= |V(F_{2})|= f(G)+ f(H)$. Thus $G\cup H$ is well-f-covered.

Conversely, Let $G\cup H$ be well-f-covered. We show that $G$ and $H$ are both well-f-covered, and $f(G\cup H)= f(G)+ f(H)$. Let $F_{1G}$ and $F_{2G}$ be two maximal forests in $G$. Suppose that $F_{H}$ be a maximal forests in $H$. The subgraphs $F_{1G}\cup F_{H}$ and $F_{2G}\cup F_{H}$ are two maximal forests in $G\cup H$. Since $G\cup H$ is well-f-covered, $|V(F_{1G}\cup F_{H})|= |V(F_{2g}\cup F_{H})|$. Therefore $|V(F_{1G})|= |V(F_{2G})|$, and thus $G$ is well-f-covered. A similar argument shows that $H$ is well-f-covered, too. Now $f(G\cup H)= |V(F_{1G}\cup F_{H})|= |V(F_{1G})|+ |V(F_{H})|= f(G)+ f(H)$.
\end{proof}

Upon theorem 6.1, if $G$ be a well-f-covered graph, adding or removing an isolated vertex does not invalidate the well-f-coveredness of $G$, and only adds 1 or -1 to the forest number of $G$, respectively. Now upon what was told about bridges, it follows that adding or removing a pendant edge does not invalidate well-f-coveredness and adds 1 or -1 the forest number, respectively. Therefore for checking the well-f-coveredness of any graph, one can first remove all isolated vertices, bridges, and pendant edges. Also for two disjoint well-f-covered graphs $G$ and $H$, if we add a bridge between a vertex of $G$ and a vertex of $H$, the obtained graph will be well-f-covered and its forest number equals $f(G)+ f(H)$.

But it is not necessary that the graph obtained by taking two disjoint well-f-covered graphs $G$ and $H$, and identifying a vertex $x$ of $G$ with a vertex $y$ of $H$, be well-f-covered. Consider two $K_{3}$, and identify a vertex of one with a vertex of the other. The result is a graph which is not well-f-covered.
\begin{center}
\definecolor{ududff}{rgb}{0.30196078431372547,0.30196078431372547,1.}
\begin{tikzpicture}[line cap=round,line join=round,>=triangle 45,x=1.0cm,y=1.0cm]
\clip(0.,2.) rectangle (3.3,5.4);
\draw [line width=1.2pt] (0.74,3.78)-- (1.72,4.36);
\draw [line width=1.2pt] (1.72,4.36)-- (0.74,4.92);
\draw (0.9,2.96) node[anchor=north west] {\small{Figure 4}};
\draw [line width=1.2pt] (0.74,4.92)-- (0.74,3.78);
\draw [line width=1.2pt] (1.72,4.36)-- (2.76,4.92);
\draw [line width=1.2pt] (2.76,4.92)-- (2.76,3.72);
\draw [line width=1.2pt] (2.76,3.72)-- (1.72,4.36);
\begin{scriptsize}
\draw [fill=ududff] (0.74,3.78) circle (1.5pt);
\draw [fill=ududff] (1.72,4.36) circle (1.5pt);
\draw [fill=ududff] (0.74,4.92) circle (1.5pt);
\draw [fill=ududff] (2.76,4.92) circle (1.5pt);
\draw [fill=ududff] (2.76,3.72) circle (1.5pt);
\end{scriptsize}
\end{tikzpicture}
\end{center}
\label{Figure 4}
Under above conditions, if $G$ has a maximal forest $F$ which does not contain $x$, and $H$ has a maximal forest $F^{\prime}$ which does not contain $y$, the result graph is not well-f-covered because if $S$ be a maximal forest of $G$ containing $x$, and $S^{\prime}$ be a maximal forest of $H$ containing $y$, the orders of two maximal forests $F\cup H$ and $S\cup S^{\prime}$ are not equal.

\begin{lem}
Let $G$ and $H$ be two disjoint well-f-covered graphs and $x$ be a vertex in $G$ such that every maximal forest of $G$ contains $x$. Then for any vertex $y$ in $H$, the graph obtained by identifying $x$ and $y$ is well-f-covered and its forest number is $f(G)+ f(H)- 1$.
\end{lem}

\begin{proof}
Let $F_{1}$ and $F_{2}$ be two maximal forests in the result graph. The graphs induced by the sets $V(F_{1})\cap V(G)$ and $V(F_{2})\cap V(G)$ in $G$, both are maximal forest and therefore both contain $x$. Thus $F_{1}$ and $F_{2}$ both contain the vertex $x$. Since $G$ is well-f-covered, $|V(F_{1})\cap V(G)|= |V(F_{2})\cap V(G)|$. Also the graphs induced by $V(F_{1})\cap V(H)$ and $V(F_{2})\cap V(H)$ in $H$, are maximal forest and therefore $|V(F_{1})\cap V(H)|= |V(F_{2})\cap V(H)|$. Thus $|V(F_{1})|= |V(F_{2})|= |V(F_{1})\cap V(G)|+ |V(F_{1})\cap V(H)|- 1= f(G)+ f(H)- 1$.
\end{proof}

As an application of Lemma 6.2, one can take a well-f-covered graph and identifies a vertex of it with a vertex of any forest to obtain a larger well-f-covered graph.

\begin{cor}
Consider two disjoint well-f-covered graphs $G$ and $H$. Suppose that the graph $A$ is obtained by taking $G$ and $H$, and connecting a vertex of $G$ and a vertex of $H$ by a path of length $d$ with new interior vertices. The graph $A$ is well-f-covered and $f(A)= f(G)+ f(H)+ d- 1$.
\end{cor}

\begin{proof}
Let the made path be:
\[(x=) v_{0} v_{1} \ldots v_{d-1} v_{d}(=y)\]
where $d(\geq 1)$ be the length of path. Suppose that $e_{i}$ be the edge between $v_{i-1}$ and $v_{i}$ $(1\leq i\leq d)$. Add the pendant edge $e_{1}$ to $G$ to obtain a well-f-covered graph, and similarly go on with adding the pendant edges $e_{2}, \ldots, e_{d}$. Call the result graph, $G^{\ast}$. The graph $G^{\ast}$ is well-f-covered and $f(G^{\ast})= f(G)+ d$. Every maximal forest of $G^{\ast}$ contains $y$. Now use Lemma 6.2 to obtain the final graph by identifying the vertex $v_{d}$ in the recent graph, with the vertex $y$ in $H$.
\end{proof}

\begin{lem}
Consider a graph $G$ containing an edge $e$ which is incident with a vertex $y$ of degree 2. Let $G^{\ast}$ be the graph obtained from $G$ by replacing the edge $e$ by a path of length $d(\geq 2)$ with new interior vertices. Then $G^{\ast}$ is well-f-covered, if and only if $G$ be well-f-covered, and anyway $f(G^{\ast})= f(G)+ d- 1$.
\end{lem}

\begin{proof}
Let $G$ be well-f-covered and the mentioned path be:
\[(y=) v_{0} v_{1} \ldots v_{d-1} v_{d}\]
where $v_{d}$ is a vertex of $G$. Consider an arbitrary maximal forest $F^{\ast}$ in $G$. If $F^{\ast}$ does not contain the vertex $v_{d}$, then by maximality, it contains all vertices $v_{0}, v_{1}, \ldots, v_{d-1}$, and if it contains $v_{d}$, it contains at least $d-1$ vertices of $v_{0}, v_{1}, \ldots, v_{d-1}$. In the first case, by removing all vartices $v_{1}, \ldots, v_{d-1}$, we can obtain a maximal forest for $G$, and in the second case, if all $v_{1}, \ldots, v_{d-1}$ are contained in $F^{\ast}$, by removing these vertices from $F^{\ast}$, we can obtain a maximal forest for $G$, and if one of $v_{1}, \ldots, v_{d-1}$ as $v_{j}$ is not contained in $F^{\ast}$, then $y$ is in $F^{\ast}$ and therefore by removing $y$ and all vertices $v_{i}, \ 1\leq i(\neq j)\leq d-1$, we can obtain a maximal forest for $G$, again. Thus in any case, one can obtain a maximal forest for $G$ by removing exactly $d-1$ vertices from $f^{\ast}$. Since every maximal forest of $G$ is of order $f(G)$, $|V(f^{\ast})|= f(G)+ d- 1$.

Conversely, suppose that $G^{\ast}$ be well-f-covered. Let $F$ be a maximal forest in $G$. The graph induced by $V(F)\cup \{v_{1}, \ldots, v_{n-1}\}$ is a maximal forest in $G^{\ast}$ and therefore its order is $f(G^{\ast})$. Thus $|V(F)|= f(G^{\ast})- (d-1)$. Therefore $G$ is well-f-covered.
\end{proof}

\begin{lem}
Suppose that $G$ and $H$ are two disjoint graphs such that $H$ is a forest and suppose that $v_{1}, \ldots, v_{n-1}, v_{n}$ be a path of length at least 1 in $G$, and $w_{1}, \ldots, w_{n-1}, w_{n}$ ba a path of the same length in $H$. Let the graph $L$ is obtained by identifying each $v_{i}$ with $w_{i}$, $1\leq i\leq n$. The graph $L$ is well-f-covered, if and only if $G$ be well-f-covered, and anyway $f(L)= f(G)- |V(H)|+ n$.
\end{lem}

\begin{proof}
First, we prove that all vertices of any cycle of the graph $L$, lies in $G$. Let $C$ be a cycle in $L$. Since $H$ is a forest, $C$ contains at least one vertex of $G$ as $x$. Let the vertices of $C$ be:
\[(x=) x_{1} x_{2} \ldots x_{m-1} x_{m} (=x_{1})\]
where $m\geq 3$. Suppose that at least one of $x_{2}, \ldots, x_{m-1}$ is not in $G$ (on the contrary). Let $i$ be the smallest index which $x_{i}\notin V(G)$. Upon the structure of $L$, The vertex $x_{i-1}$ is one of the identifying vertices. Let this vertex be $w_{t}$ in $H$. Suppose that $j$ be the minimum index which $j\geq i$ and $x_{j}\in V(G)$ (note that there is such a vertex because $x_{m}$ is in $G$). The vertex $x_{j}$ is another identifying vertex, say $w_{t^{\prime}}$ $(t^{\prime}\neq t)$. But in this case, the cycle:
\[x_{i-1} x_{i} \ldots x_{j}(= w_{t^{\prime}})\ldots w_{t}(= x_{i-1})\]
is contained in $H$, a contradiction.

Now, let $G$ be well-f-covered. Suppose that $F^{\ast}$ be a maximal forest in $L$. By maximality, $F^{\ast}$ contains all vertices of $H$. One can see that $G[V(F^{\ast})- (V(H)- \{w_{1}, \ldots, w_{n-1}, w_{n}\})]$ is a maximal forest of $G$ and therefore its order is $f(G)$. Therefore $|V(F^{\ast})|= f(G)+ |V(H)|- n$. Thus $L$ is well-f-covered.

Conversely, let $L$ be well-f-covered. Let $F$ be a maximal forest in $G$. The subgraph $L[V(F)\cup (V(H)- \{w_{1}, \ldots, w_{n-1}, w_{n}\})]$ is a maximal forest of $L$. Since $L$ is well-f-covered, $|V(F)\cup (V(H)- \{w_{1}, \ldots, w_{n-1}, w_{n}\})|$ is $f(L)$. Therefore $|V(F)|= f(L)- |V(H)|+ n$. Thus $G$ is well-f-covered.
\end{proof}

Now we obtain some larger well-f-covered graphs from complete graphs.

\begin{thm}
Suppose that $G$ be a graph obtained by taking a complete graph $K_{n}$ with $n\geq 2$, and for every edge $e$ in $E(K_{n})$, adding a vertex $u_{e}$ which is adjacent with both endpoints of $e$. Then $G$ is well-f-covered with forest number $C(n,2)+ 1$.
\end{thm}

\begin{proof}
Let $F$ be a maximal forest of $G$. By maximality, $F$ contains at least one of the vertices of the part $K_{n}$ (note that no two distinct vertices are adjacent in $G$). Call this vertex $x$. If $F$ does not contain another vertex of $K_{n}$, it contains all added vertices. Otherwise, $F$ contains only one another vertex of $K_{n}$, say $y$. Let $e^{\ast}= \{x,y\}$, be the edge between $x$ and $y$ in $K_{n}$. By maximality, $F$ contains all vertices $u_{e}$ which $e\neq e^{\ast}$. In both cases, $|F|= C(n,2)+ 1$.
\end{proof}

\begin{thm}
Let $G$ be a graph obtained by taking $s$ disjoint complete graphs of the same order $n$, and adding edges between the cliques so that the added edges form a matching. Then $G$ is well-f-covered and if $n\geq 3$, then $f(G)= 2s$.
\end{thm}

\begin{proof}
If $n=1$, every edge of $G$ is an element of the matching and therefore $G$ is a forest, itself. Let $n\geq 2$. Let we denote the complete graphs which form $G$, with $A_{1}, \ldots, A_{s}$, respectively. There are two cases. 

$\mathbf{Case 1.}$ $n> 2$. In this case, we show that every maximal forest of $G$ contains exactly two vertices from any $A_{i}, \ 1\leq i\leq s$. Let $F$ be a maximal forest of $G$. For any $1\leq i\leq s$, $F$ contains at most two vertices from $A_{i}$ because every tree distinct vertices of $A_{i}$ form a cycle. Since $F$ is maximal and every vertex of $A_{i}$ is incident with at most one element of  the matching, $F$ contains at least one vertex of $A_{i}$, say $x$. Suppose that $A_{i}$ does not contain any other vertex of $A_{i}$ (on the contrary). The graph $A_{i}$ has two distinct vertices except $x$, say $y$ and $z$. Since $y$ is not contained in $F$, the graph $G[V(F)\cup \{y\}]$ has a cycle $C_{y}$ which contains $y$ as a vertex. Similarly, the graph $G[V(F)\cup \{z\}]$ has a cycle as $C_{z}$ which contains $z$. The cycles $C_{y}$ and $C_{z}$ contain $x$, as well because no two elements of matching are adjacent. Therefore the edge between $x$ and $y$ in $A_{i}$ is an edge of $C_{y}$, and the edge between $x$ and $z$ in $A_{i}$ is an edge of $C_{z}$. Let $e_{x}$ and $e_{y}$ be the edges in $C_{y}$, except the edge $\{x,y\}$, which are incident with $x$ and $y$, respectively. Also let $e^{\prime}_{x}$ and $e_{z}$ be the edges in $C_{z}$, except the edge $\{x,z\}$, which are incident with $x$ and $z$, respectively. Since $C_{y}$ does not contain any vertices of $A_{i}$ expect $x$ and $y$, the edges $e_{x}$ and $e_{y}$ are in the matching. Similarly $e^{\prime}_{x}$ and $e_{z}$ are in the matching. Therefore $e_{x}= e^{\prime}_{x}$ (because $x$ is incident at most with one element of the matching).
\begin{center}
\definecolor{ududff}{rgb}{0.30196078431372547,0.30196078431372547,1.}
\begin{tikzpicture}[line cap=round,line join=round,>=triangle 45,x=1.0cm,y=1.0cm]
\clip(-2.8,1.3) rectangle (2.8,5.6);
\draw [line width=1.2pt] (-2.28,3.8)-- (-1.54,2.82);
\draw (-0.64,1.98) node[anchor=north west] {\small{Figure 5}};
\draw [line width=1.2pt] (-1.56,4.78)-- (-2.28,3.8);
\draw [line width=1.2pt] (2.,2.82)-- (-1.54,2.82);
\draw [line width=1.2pt] (2.02,3.78)-- (-2.28,3.8);
\draw [line width=1.2pt] (-1.56,4.78)-- (2.04,4.74);
\draw (-2.75,4.02) node[anchor=north west] {$\mathit{x}$};
\draw (-1.99,5.08) node[anchor=north west] {$\mathit{y}$};
\draw (-2.02,3.03) node[anchor=north west] {$\mathit{z}$};
\draw (2.,5.15) node[anchor=north west] {$\mathit{y^{'}}$};
\draw (1.97,4) node[anchor=north west] {$\mathit{u_{1}}$};
\draw (1.99,3.21) node[anchor=north west] {$\mathit{z^{'}}$};
\draw (-0.14,5.23) node[anchor=north west] {$\mathit{e_{y}}$};
\draw (-0.12,4.2) node[anchor=north west] {$\mathit{e_{x}}$};
\draw (-0.14,3.25) node[anchor=north west] {$\mathit{e_{z}}$};
\begin{scriptsize}
\draw [fill=ududff] (-2.28,3.8) circle (1.5pt);
\draw [fill=ududff] (-1.54,2.82) circle (1.5pt);
\draw [fill=ududff] (-1.56,4.78) circle (1.5pt);
\draw [fill=ududff] (2.02,3.78) circle (1.5pt);
\draw [fill=ududff] (2.,2.82) circle (1.5pt);
\draw [fill=ududff] (2.04,4.74) circle (1.5pt);
\end{scriptsize}
\end{tikzpicture}
\end{center}
\label{Figure 5}
Suppose that $y^{\prime}(\neq y)$, $u_{1}(\neq x)$, and $z^{\prime}(\neq z)$ be the vertices incident with $e_{y}$, $e_{x}$, and $e_{z}$, respectively. These tree vertices are distinct pairwise. Suppose that $y^{\prime}$ be in $A_{l}$, $u_{1}$ be in $A_{r}$, and $z^{\prime}$ be in $A_{t}$, where $l, r, t\neq j$. At least two item of $l, r$, and $t$ are distinct because $F$ can not contain tree distinct vertices from the same clique. On the other hand, the edge of $C_{y}$, expect $e_{x}$, which is incident with $u_{1}$, is not an element of the matching and therefore is contained in $A_{r}$. Similarly, the edge of $C_{z}$, expect $e_{x}$, which is incident with $u_{1}$, is contained in $A_{r}$. Since $F$ can contain at most two vertices of $A_{r}$, two recent edges are equal. Let the other vertex (except $u_{1}$) incident with this common edge be $u_{2}$, a vertex in $A_{r}$. Now the other edge incident with $u_{2}$ in $C_{y}$, also the other edge incident with $u_{2}$ in $C_{z}$, are elements of the matching and therefore are the same. This common edge is incident with a vertex in a clique except $A_{j}$ and $A_{r}$. Now assuming this edge instead of $e_{x}$, and repeating the above procedure, $C_{y}$ and $C_{z}$ will be identified, a contradiction (with $y^{\prime} \neq z^{\prime}$).

$\mathbf{Case 2.}$ $n=2$. In this case, the degree of any vertex in $G$ is at most two. One can see that each connected components of $G$ is a cycle or a tree. Thus $G$ is well-f-covered upon theorem 6.1.
\end{proof}

Now, we consider join.

\begin{thm}
Let $G$ and $H$ be two disjoint nonempty graphs, and $G$ be without any maximal independent set of size 1. Then the join of $G$ and $H$ is well-f-covered, if and only if:

\hskip -0.45 true cm(1) $G$ and $H$ be well-f-covered,

\hskip -0.45 true cm(2) $G$ be well-covered and any two maximal independent vertex set of size at least 2 in $H$ be of the same size.

\hskip -0.45 true cm(3) $f(G)= f(H)= \alpha (G)+ 1= \alpha (H)+ 1$,

\hskip -0.45 true cm and anyway $f(G\vee H)= f(G)= f(H)$.
\end{thm} 

\begin{proof}
Let the join of $G$ and $H$ is well-f-covered. Suppose that $F_{G}$ be a maximal forest for $G$. Since $G$ is not empty, $F_{G}$ contains two vertices, $x$ and $y$, adjacent in $G$ (note that by maximality, $F_{G}$ is not a trivial graph, and if $F_{G}$ be an empty graph with $n$ $(n\geq 2)$ vertices, for a non-isolated vertex $g$ of $G$, one can add $g$ to $F_{G}$ obtaining a larger forest in $G$, a contradiction). The graph $F_{G}$ is a maximal forest of $G\vee H$ because any vertex of $H$ is adjacent with both $x$ and $y$. Therefore $|V(H)|= f(G\vee H)$. Thus $G$ is well-f-covered and $f(G)= f(G\vee H)$. A similar argument shows that $H$ is well-f-covered too and $f(H)= f(G\vee H)$. Let $B_{G}$ be a maximal independent set in $G$. The set $B_{G}$ has two distinct vertices, say $u_{1}$ and $u_{2}$. Let $v_{1}$ be any vertex of $H$. The subgraph $(G\vee H) [B_{G}\cup \{V_{1}\}]$ is a maximal forest in $G\vee H$ because first, every vertex of it, except $v_{1}$, is of degree 1, and therefore it is without any cycle, and secondary, adding any vertex $t\in V(G)$ to it, forms a cycle (note that $t$ is adjacent with a vertex $u$ of $G$, and both $t$ and $u$ are adjacent with $v_{1}$) while adding any vertex $v_{2}\in V(H)$ forms the cycle $u_{1} v_{1} u_{2} v_{2}$. Thus $|B_{G}|+ 1= f(G\vee H)$ and therefore $G$ is well-covered and $\alpha (G)+ 1= f(G\vee H)$. Now let $S_{1}$ and $S_{2}$ be two maximal independent sets of size at least 2 in $H$. The subgraphs $(G\vee H) [S_{1}\cup \{g\}]$ and $(G\vee H) [S_{2}\cup \{g\}]$, where $g\in V(G)$, are two maximal forests of $G\vee H$. Thus $|S_{1}\cup \{g\}|= |S_{2}\cup \{g\}|$ and therefore $|S_{1}|=|S_{2}|$. Finally, we prove that  $f(H)= \alpha (H)+1$. If $H$ be a complete graph, since $H$ is not empty, $f(H)=2$ and $\alpha (H)=1$, and therefore the equality holds. Otherwise, $H$ has two non-adjacent vertices $r,s$. Let $B_{H}$ be a maximal independent set in $H$ containing $\{r,s\}$. The subgraph $(G\vee H) [B_{H}\cup \{g\}]$ where $g\in V(G)$, is a maximal forest of $G\vee H$. Thus $\alpha (H)+ 1= f(G\vee H)= f(H)$.

Conversely, suppose that conditions (1), (2), and (3) hold. We show that $G\vee H$ is well-f-covered and $f(G\vee H)= f(G)= f(H)$. 

Let $F^{\ast}$ be a maximal forest of $G\vee H$. If all vertices of $F^{\ast}$ be in $G$, then $F^{\ast}$ is a maximal forest of $G$ and therefore by (1), $|V(F^{\ast})|= f(G)$. Similarly, if all vertices of $F^{\ast}$ be in $H$, $|V(F^{\ast})|= f(H)$. Suppose both $V(F^{\ast})\cap V(G)$ and $V(F^{\ast})\cap V(H)$ ba nonempty. Let $X:= V(F^{\ast})\cap V(G)$ and $Y:= V(F^{\ast})\cap V(H)$. The sets $X$ and $Y$ are independent set in $G$ and $H$, respectively. On the other hand, at least one of $X$ and $Y$ is of size 1. If $|Y|=1$, then $X$ is a maximal independent set of $G$ and therefore by (2), $|V(F^{\ast})|= \alpha (G)+ 1$. If $|X|=1$, then $Y$ is a maximal independent vertex set in $H$ of size at least 2, and therefore by (2), $|V(F^\ast)|= \alpha (H)+ 1$. By (3), in any case, $|V(F^{\ast})|= f(G)= f(H)$.
\end{proof}

\begin{rem}
In Theorem 6.8, if in addition, the graph $H$ be without any maximal independent set of size 1, then the well-coveredness of $H$ is implied, as well.
\end{rem}

\begin{thm}
Let $G$ and $H$ be two disjoint graphs, each one has a maximal independent set of size 1. Then $G\vee H$ is well-f-covered, if and only if $G$ and $H$ be complete graphs.
\end{thm}

\begin{proof}
If both $G$ and $H$ be complete, then $G\vee H$ is complete, and therefore it is well-f-covered.

Conversely, Suppose that $G\vee H$ be well-f-covered. Let at least one of $G$ and $H$, give $G$, is not complete (on the contrary). $G$ has two non-adjacent vertices $x$ and $y$. Let $M_{G}$ be a maximal independent vertex set in $G$, containing $\{x,y\}$. Also let $\{g\}$ and $\{h\}$ be maximal independent set of size 1 in $G$ and $H$, respectively. The subgraphs $(G\vee H) [M_{G}\cup \{b\}]$ and $(G\vee H) [\{a,b\}]$ are two maximal forests of different sizes in $G\vee H$.
\end{proof}

\begin{thm}
Let $G$ and $H$ be two disjoint graphs such that $G$ is nonempty and $H$ is empty of size $n\geq 2$. Then $G\vee H$ is well-f-covered if and only if:

\hskip -0.45 true cm(1) $G$ be well-f-covered,

\hskip -0.45 true cm(2) all maximal independent sets of size at least 2 in $G$ are of the same size,

\hskip -0.45 true cm(3) $f(G)= \alpha (G)+ 1= n+1$, 

\hskip -0.45 true cm and anyway $f(G\vee H)= f(G)= n+1$.
\end{thm}
\begin{proof}
Let $G\vee H$ be well-f-covered. By an argument similar to what we considered in Theorem 6.8, it is shown that $G$ is well-f-covered and $f(G)= f(G\vee H)$. For a vertex $g$ in $G$, the subgraph $(G\vee H) [V(H)\cup \{g\}]$ in $G\vee H$ is a maximal forest. Thus $n+1= f(G\vee H)$. Finally, with an argument similar to what we considered about $H$ in Theorem 6.8, it can be seen that every two maximal independent  set of size at least 2 in $G$ are of the same size and $\alpha (H)+ 1= f(H)$.

Conversely, let conditions (1), (2), and (3) hold. We prove that the graph $G\vee H$ is well-f-covered and $f(G\vee H)= f(G)= n+1$. Let $F^{\ast}$ be any maximal forest in $G\vee H$. By maximality. $F^{\ast}$ has at least one vertex from $G$. If $V(F^{\ast})\subseteq V(G)$, then $F^{\ast}$ is a maximal forest of $G$ and therefore by (1), $|V(F^{\ast})|= f(G)$. Otherwise, $X:= V(F^{\ast})\cap V(G)$ and $Y:= V(F^{\ast})\cap V(H)$ are nonempty independent set in $G$ and $H$, respectively. At least one of $X$ and $Y$ is of size 1. Now, if $|X|=1$, then $Y$ is a maximal independent  set in $X$ and therefore is equal to $V(H)$; Thus $|V(F^\ast)|= n+1$. If $|Y|=1$, then $X$ is a maximal independent set in $G$ of size at least 2 (note that $n\geq 2$) and therefore by (2), $|V(F^{\ast})|= \alpha (G)+1$. Now by (3), in any case, $|V(F^{\ast})|= f(G)= n+1$.
\end{proof}

\begin{rem}
In Theorem 6.11, if in addition, the graph $G$ be without any maximal independent set of size 1, then the well-coveredness of $G$ is implied, as well.
\end{rem}

In case where one of the components of $f\vee H$ be trivial graph, if the other component has maximal independent set of size 1, Theorem 6.10 can be used, but if the other graph be without any such independent set, the following theorem can be used.

\begin{thm}
Let $G$ be a nonempty graph without any maximal independent
set, and $H$ be a trivial graph, disjoint from $G$. Then $G\vee H$ is well-f-covered, if and only if $G$ be both well-f-covered and well-covered, and $f(G)= \alpha (G)+1$, and anyway $f(G\vee H)= f(G)$.
\end{thm}

\begin{proof}
Suppose that $G\vee H$ is well-f-covered. An argument similar to what was considered in Theorem 6.8, Shows that $G$ is well-f-covered with forest number $f(G\vee H)$ and $G$ is well-covered and $f(G)= \alpha (G)+ 1$.

Conversely, let $G$ is well-f-covered and well-covered, and also $f(G)= \alpha (G)+ 1$. Suppose that $V(H)= \{v\}$. Consider a maximal forest $F^{\ast}$ in $G\vee H$. If $v\notin V(F^{\ast})$, then $F^{\ast}$ is a maximal forest of $G$ and therefore $|V(F^{\ast})|= f(G)$. Otherwise, $V(F^{\ast})- \{v\}$ is a maximal independent set in $G$ and therefore since $G$ is well-covered, $|V(F^{\ast})|= \alpha (G)+ 1$. Now since $f(G)= \alpha (G)+ 1$, the graph $G\vee H$ is well-f-covered and $f(G\vee H)= f(G)$.
\end{proof}

\begin{cor}
A  wheel of order $n$, $W_{n}$ $(n\geq 4)$, is well-f-covered, if and only if $n\in \{4,5\}$.
\end{cor}

\begin{proof}
Any $W_{n}$ $(n\geq 4)$ is the join of a cycle of $n$ vertices, $C_{n}$, with a trivial graph disjoint from it. $C_{n}$ is well-f-covered with forest number $n-2$, and it is well-covered with independence number $[\frac{n-1}{2}]$. But for any $n\geq 4$, $[\frac{n-1}{2}]= (n-2)-1$, if and only if $n\in \{4,5\}$ (the equality is not satisfied for $n=6$, and for $n\geq 7$, $[\frac{n-1}{2}]< \frac{n-1}{2}+ 1\leq n-3$). Now use Theorem 6.13.
\end{proof}

Finally, we consider the case where both components of $G\vee H$ are empty.

\begin{thm}
Consider two disjoint empty graphs $G$ and $H$, of order $n_{1}$ and $n_{2}$, respectively. $G\vee H$ is well-f-covered, if and only if $n_{1}= n_{2}$ or at least one of $n_{1}$ and $n_{2}$ be 1.
\end{thm}

\begin{proof}
Let $G\vee H$ be well-f-covered. If $|V(G)|, |V(H)|\geq 2$, for vertices $u\in V(G)$ and $v\in V(H)$, the subgraphs $(G\vee H) [V(G)\cup \{u\}]$ and  $(G\vee H) [V(H)\cup \{v\}]$ are two maximal forests in $G\vee H$. Thus $n_{1}= n_{2}$.

Conversely, if $n_{1}= 1$ or $n_{2}= 1$, $G\vee H$ is forest and therefore is well-f-covered. Let $n_{1}= n_{2}$, and let $F^{\ast}$ be a maximal forest of $G\vee H$. By maximality, $|V(F^{\ast})\cap V(G)|\geq 1$. If $|V(F^{\ast})\cap V(G)|= 1$, then $F^{\ast}$ contains all vertices of $H$ and therefore is of size $n_{2}+ 1$. If $|V(F^{\ast})\cap V(G)|\geq 2$, then $|V(F^{\ast})\cap V(H)|\leq 1$ because $F$ is cycleless. In this case, by maximality, $|V(F^{\ast})\cap V(H)|= 1$ and $|V(F^{\ast})\cap V(G)|= n_{1}$, and therefore $|V(F^{\ast})|= n_{1}+ 1$.
\end{proof}

The following characterization is implied from Theorem 6.14, immediately.

\begin{cor}
A complete bipartite graph $K_{r,s}$ is well-f-covered, if and only if $r=s$ or at least one of $r$ and $s$ be 1.
\end{cor}

\vskip 0.4 true cm


\vskip 0.4 true cm




\end{document}